\documentclass[11pt]{article}
       \usepackage{amsfonts}
       \usepackage{stmaryrd}
       \usepackage{latexsym,amssymb,mathrsfs,fancyhdr}
       \font\tenmsb=msbm10
       \font\sevenmsb=msbm7
       \font\fivemsb=msbm5
       \catcode`\@=11
       \ifx\amstexloaded@\relax\catcode`\@=\active
       \endinput\else\let\amstexloaded@\relax\fi
       \def\spaces@{\space\space\space\space\space}
       \def\spaces@@{\spaces@\spaces@\spaces@\spaces@\spaces@}
       \def\space@.  {\futurelet\space@\relax}
       \space@.   %
       \def\Err@#1{\errhelp\defaulthelp@\errmessage{AmS-TeX error: #1}}
       \def\relaxnext@{\let\next\relax}
       \def\accentfam@{7}
       \def\noaccents@{\def\accentfam@{0}}
       \def\Cal{\relaxnext@\ifmmode\let\next\Cal@\else
       \def\next{\Err@{Use \string\Cal\space only in math mode}}\fi\next}
       \def\Cal@#1{{\Cal@@{#1}}}
       \def\Cal@@#1{\noaccents@\fam\tw@#1}
       \def\Bbb{\relaxnext@\ifmmode\let\next\Bbb@\else
       \def\next{\Err@{Use \string\Bbb\space only in math mode}}\fi\next}
       \def\Bbb@#1{{\Bbb@@{#1}}}
       \def\Bbb@@#1{\noaccents@\fam\msbfam#1}
       \newfam\msbfam
       \textfont\msbfam=\tenmsb
       \scriptfont\msbfam=\sevenmsb
       \scriptscriptfont\msbfam=\fivemsb
\usepackage{dsfont}
\usepackage[german,english]{babel}
\usepackage{amsmath,amssymb}
\usepackage[square, comma, sort&compress, numbers]{natbib}
\usepackage{cases}
\usepackage{mathrsfs}
\usepackage{amsmath}
\usepackage{amsthm}
\usepackage{amsfonts}
\usepackage{amssymb}
\usepackage{latexsym}
\usepackage{fancyhdr}
\usepackage{geometry}
\newtheorem{thm}{Theorem}[section]
\newtheorem{prop}[thm]{Proposition}
\newtheorem{lem}[thm]{Lemma}
\newtheorem{rem}[thm]{Remark}
\newtheorem{iteration lemma}[thm]{iteration Lemma}
\newtheorem{cor}[thm]{Corollary}

\newtheorem*{acknowledgements*}{ACKNOWLEDGEMENTS}

\begin{document}

\setlength{\columnsep}{5pt}
\title{\bf The reverse order law of the $(b, c)$-inverse in rings }
\author{ Yuanyuan Ke$^{a}$\footnote{ E-mail:  keyy086@126.com },
\ Dijana Mosi\'{c}$^{b}$\footnote{ Corresponding author. E-mail: dijana@pmf.ni.ac.rs}, \ Jianlong Chen$^{a}$\footnote{ E-mail: jlchen@seu.edu.cn },
\\\\
$^{a}$Department of  Mathematics, Southeast University,  Nanjing 210096,  China\\
$^{b}$Faculty of Sciences and Mathematics, University of Ni\v{s}, Vi\v{s}egradska 33,\\ 18000 Ni\v{s}, Serbia
}
     \date{}

\maketitle
\begin{quote}
{\textbf{Abstract}: \small We present equivalent conditions of
reverse order law for the $(b, c)$-inverse
$(aw)^{(b,c)}=w^{(b,s)}a^{(t,c)}$ to hold in a ring. Also, we
study various mixed-type reverse order laws for the $(b,
c)$-inverse. As a consequence, we get results related to the
reverse order law for the inverse along an element. More general
case of reverse order law $(a_1a_2)^{(b_3, c_3)}=a_2^{(b_2,
c_2)}a_1^{(b_1, c_1)}$ is considered too.



}

\end{quote}

\section{Introduction}
    Throughout the  paper, we assume that $R$ is a ring with identity. We use $R^\bullet$ to denote
the set of all idempotents ($p^2=p$) of $R$. 

     Let $b, c\in R$.
    The concept of the $(b, c)$-inverse as a generalization of the Moore-Penrose inverse, the Drazin inverse, the Chipman's weighted
    inverse and the Bott-Duffin inverse, was for the first time  introduced by Drazin  in 2012 \cite{D} in the settings of
    rings.  Recall that an element $a\in R$ is said to be $(b, c)$-invertible if there exists $y\in R$ such that
    $$y\in (bRy)\cap (yRc),\quad\quad yab=b, \quad\quad cay=c.$$
    If such $y$ exists, it is unique and it is called the $(b, c)$-inverse of $a$, denoted by $a^{(b, c)}$.
    The set of all $(b, c)$-invertible elements of $R$ will be denoted by $R^{(b, c)}$.
    For more results of the $(b, c)$-inverse, we refer the reader to see \cite{Drazin-LMA2014,Drazin-LAA14-Uniqueproof,Ke,W-C-C}.

    Drazin  \cite{D} introduced an outer generalized inverse relative to a pair of idempotents $e, f\in R$
   which intermediates between the Bott-Duffin inverse and the $(b,
   c)$-inverse. This generalized inverse is called  the Bott-Duffin $(e, f)$-inverse.
    Recall that the Bott-Duffin $(e, f)$-inverse of $a\in R$ is the element $y\in R$ which satisfies
    $$
    y=ey=yf, \quad yae=e, \quad fay=f.
    $$
    If the Bott-Duffin $(e, f)$-inverse of $a$ exists, it is unique and denoted by $a^{BD}_{e, f}$.
    According to Kant\'{u}n-Montiel \cite[Proposition 3.4]{Kantun-Montiel-LMA14}, an image-kernel $(e, f)$-inverse is the
    Bott-Duffin $(e, 1-f)$-inverse.  More results on the Bott-Duffin $(e, f)$-inverse and image-kernel $(p, q)$-inverse can be found in \cite{Ke-LAA16,MD-charIK,MDjKM,MDj-innerIK}.

Mary in \cite{Mary2011} introduced a new generalized inverse,
called the inverse along element. An element $a\in R$ is said to
be invertible along $d\in R$ (or Mary invertible) if there exists
$y\in R$ such that
$$yad=d=day, \quad yR \subseteq dR, \quad Ry\subseteq Rd.$$
If such $y\in R$ exists, it is unique and  will be denoted by
$a^{\parallel d}$. We use $R^{\| d}$ to denote the all Mary invertible elements in $R$. This inverse unify some well-known generalized
inverses, such as the group inverse, Drazin inverse and
Moore-Penrose inverse.  Also, the inverse along element $d$ is a
special case of $(b, c)$-inverse, for $(b, c)=(d, d)$
\cite[Proposition 6.1]{D}. Several authors also have studied this
new outer inverse (see
\cite{Benitez-Boasso-1507,Benitez-Boasso-1509,Mary-P
2012,Zhu-Chen-Patricio-LMA15104,Zhu-Patricio-Chen-Zhang-LMA15105}).

We recall that a $(p,q)$-outer generalized inverse of $a$ with
prescribed idempotents $p$ and $q$ introduced by Djordjev\'{c} and
Wei in \cite[Definition 2.1]{Djordjevic-Wei-CommAl05}. Given $p,~q
\in R^\bullet$, an element $a \in R$ has a $(p,q)$-outer
generalized inverse $y \in R$, if
$$ yay=y,\quad ya=p,\quad ay=1-q. $$
This $(p,q)$-outer inverse is unique if it exists, and we write
$y=a_{p,q}^{(2)}$.  If
$a^{(2)}_{p,q}$ satisfies $aa^{(2)}_{p,q}a=a$, then
$a^{(2)}_{p,q}=a^{(1,2)}_{p,q}$ is the $(p,q)$-reflexive
generalized inverse of $a$. The set of all $(p,q)$-outer ($(p,q)$-reflexive) generalized
invertible elements of $R$ will be denoted by $R^{(2)}_{p,q}$ ($R^{(1, 2)}_{p,q}$, resp.).

An element $a\in R$ is (von Neumann) regular if it has an inner
inverse $y$, i.e. if there exists $y\in R$ such that $aya=a$. Any
inner inverse of $a$ will be denoted by $a^-$. The set of all
regular elements of $R$ will be denoted by $R^-$.
An element $a\in R$ is called group invertible if there is $y\in R$ such that $$aya=a, \quad yay=y, \quad ay=ya.$$
This $y$ also is unique when it exists, and we denote it as $a^\#$. The set of all group invertible elements in $R$ is denoted by $R^\#$.

\begin{lem}\emph{\cite[Theorem 1]{Drazin1958}}\label{ke-(b,c)-4-lem6-commut}
Let $a, d\in R$. If $a\in R^\#$ and $da=ad$, then $a^\#d=da^\#$.
\end{lem}

For an element $a \in R$, we define the following image ideals
$$aR=\{ax: x\in R \}, \quad Ra=\{xa: x\in R \},$$
and kernel ideals
$$ a^\circ=\{x\in R: ax=0 \}, \quad ~^\circ a=\{x\in R:xa=0 \}.$$
Let $a, y \in R$. Then $aR = yR$ if and only if there exist $u, v \in R$ such that $a = yu$ and $y = av$. Similarly, $Ra = Ry$ if and only if there exist $s, t \in R$ such that $a = sy$ and $y = ta$.

\section{Reverse order laws for the $(b, c)$-inverse}

It is well known that for nonsingular matrices $A$ and $B$ of the same size, we have
$$(AB)^{-1}=B^{-1}A^{-1}.$$
This equality is known as the reverse order law.

The study of the reverse order law for generalized inverses traces
back to the work of Greville \cite{Greville-SIAM1966}, who studied
the reverse order law for the Moore-Penrose inverse of a matrix.
Since then, a large amount of work has been devoted to the study
of this problem, and equivalent conditions for the Moore-Penrose
inverse reverse order law to hold have been proved in the setting
of matrices, operators, or elements of rings with involution
\cite{Hartwig-LAA1986,Djordjevic-Dincic-JMAA10,MDjfur-rolMP}.

The ``reverse order law" in a ring $R$, says that if $a, b\in R$
both invertible, then $(ab)^{-1}=b^{-1}a^{-1}$. It is a natural
question whether this has extension to generalized inverses. The
reverse order law for the generalized inverse is an useful
computational tool in applications (solving linear equations in
linear algebra or numerical analysis), and it is also interesting
from the theoretical point of view.

In this section, we investigate the reverse order law for the $(b,
c)$-inverse in rings. Precisely, we give necessary and sufficient
conditions for rules $(aw)^{(b,c)}=w^{(b,s)}a^{(t,c)}$, $(aw)^{(b,
c)}=(a^{(t, c)}aw)^{(b, s)}a^{(t, c)}$, $(aw)^{(b, c)}=w^{(b,
s)}(aww^{(b, s)})^{(t, c)}$, $(a^{(t, c)}aw)^{(b, u)}=w^{(b,
s)}a^{(t, c)}a$, $(aww^{(b, s)})^{(v, c)}=ww^{(b, s)}a^{(t, c)}$
and $(aw)^{(b, c)}=w^{(b, c)}(a^{(t, c)}aww^{(b,
c)})^{(b,c)}a^{(t, c)}$ to be satisfied. As an application of
these results, we obtain corresponding results related to the
inverse along an element. Sufficient conditions for reverse order
laws for the $(p, q)$-outer generalized inverse are presented.
Also, we consider the more general case of reverse order law
$(a_1a_2)^{(b_3, c_3)}=a_2^{(b_2, c_2)}a_1^{(b_1, c_1)}$.

The following lemmas will be useful in the sequel.


\begin{lem}\label{ke-reverse(b, c)-lem2}
Let $a, b, c\in R$. Then the following statements are equivalent:
\begin{enumerate}
\item[\rm (i)] $a^{(b, c)}$ exists;

\item[\rm (ii)] \emph{\cite[Proposition
6.1]{D}}\label{drazinLAA12-thm2.2} there exists $y\in R$ such that
$yay=y$, $yR=bR$ and $Ry=Rc$;


\item[\rm (iii)] \emph{\cite[Theorem 2.9]{Ke-dragana}} $b, c\in
R^-$, and there exists $y\in R$ such that
\begin{equation}\label{ke-reverse(b, c)-eq2}
          y=bb^-y=yc^-c, \quad b=yab, \quad  c=cay.
\end{equation}
\end{enumerate}
In this case, $y=a^{(b, c)}$.
\end{lem}

\begin{lem}\emph{\cite[Corollary 2.4, Remark 2.5]{Drazin-LMA2014}}\label{ke-(b,c)-4-rem1}
Let  $a, b, c, w\in R$. If  $a^{(b, c)}$ exists, we have
\begin{enumerate}
     \item[\rm (i)] if $ab=ba$ and $ca=ac$, then $a^{(b, c)}$ commutes with $a$;
     \item[\rm (ii)] if $ab=ba$ and $cb=bc$, then $a^{(b, c)}$ commutes with $b$;
     \item[\rm (iii)] if $ac=ca$ and $cb=bc$, then $a^{(b, c)}$ commutes with $c$.
   \end{enumerate}
\end{lem}

Specially, if $(b, c)=(d, d)$, $a$ is invertible along $d$ with $(d, d)$-inverse $a^{\| a}$ and $ad=da$, then $a^{\| d}$ commutes with $a$ and $d$.

First, we will give the necessary and sufficient condition for the
reverse order law $$(aw)^{(b, c)}=w^{(b, s)}a^{(t, c)}.$$

\begin{thm}\label{ke-reverse(b, c)-th1}
Let $a, w, b, c, s, t\in R$ be such that $a^{(t, c)}$ and $w^{(b,
s)}$ exist. Then the following statements are equivalent:
\begin{enumerate}
\item[\rm (i)] $aw\in R^{(b, c)}$ and $(aw)^{(b, c)}=w^{(b,
s)}a^{(t, c)}$;

\item[\rm (ii)] $b=w^{(b, s)}a^{(t, c)}awb$ and $c=caww^{(b,
s)}a^{(t, c)}$.
\end{enumerate}
\end{thm}
\proof
    Since $a^{(t, c)}$ exists, by Lemma \ref{ke-reverse(b, c)-lem2}, $c\in R^-$ and $a^{(t, c)}c^-c=a^{(t, c)}$. Similarly, $b\in R^-$ and $bb^-w^{(b, s)}=w^{(b, s)}$.
    Again using Lemma \ref{ke-reverse(b, c)-lem2}, $aw\in R^{(b, c)}$ and $(aw)^{(b, c)}=w^{(b, s)}a^{(t, c)}$ if and only if $b, c\in R^-$,
    $$w^{(b, s)}a^{(t, c)}=bb^-w^{(b, s)}a^{(t, c)}=w^{(b, s)}a^{(t, c)}c^-c, \quad b=w^{(b, s)}a^{(t, c)}awb,\quad c=caww^{(b, s)}a^{(t, c)}.$$
    Hence, (i) is equivalent to (ii).
\qed

If we suppose that $a$ and $w$ satisfy conditions $R a\subseteq
Rs$ and $wR\subseteq tR$, we get the next result as a consequence
of Theorem \ref{ke-reverse(b, c)-th1}.

\begin{cor}\label{ke-reverse(b, c)-cor1}
Let $a, w, b, c, s, t\in R$ be such that $a^{(t, c)}$ and $w^{(b, s)}$ exist. If $Ra\subseteq Rs$ and $wR\subseteq tR$, then $aw\in R^{(b, c)}$ and $(aw)^{(b, c)}=w^{(b, s)}a^{(t, c)}$.
\end{cor}
\proof Since $a^{(t, c)}$ and $w^{(b, s)}$ exist, by Lemma
\ref{ke-reverse(b, c)-lem2}, $s, t\in R^-$. If $Ra\subseteq Rs$,
then we get $a=xs$, for some $x\in R$, and so $a=xss^-s=as^-s$.
Similarly, the condition $wR\subseteq tR$ and $t\in R^-$ imply
$w=tt^-w$. Therefore,
$$\begin{array}{lcl}
w^{(b, s)}a^{(t, c)}awb&=&w^{(b, s)}a^{(t, c)}a(tt^-w)b=w^{(b, s)}(a^{(t, c)}att^-)wb\\
&=&w^{(b, s)}(tt^-)wb=w^{(b, s)}(tt^-w)b\\
&=&w^{(b, s)}wb=b,
\end{array}$$
and
$$\begin{array}{ccl}
caww^{(b, s)}a^{(t, c)}&=&c(as^-s)ww^{(b, s)}a^{(t, c)}=ca(s^-sww^{(b, s)})a^{(t, c)}\\
&=&cass^-a^{(t, c)}=caa^{(t, c)}=c.
\end{array}$$
Applying Theorem \ref{ke-reverse(b, c)-th1}, we can conclude that $aw\in R^{(b, c)}$ and $(aw)^{(b, c)}=w^{(b, s)}a^{(t, c)}$.
\qed

\medskip
When $s=c$ and $t=b,$ we have the following result.

%

\begin{cor}
Let $a, w, b, c\in R$ be such that $a^{(b, c)}$ and $w^{(b, c)}$ exist.
\begin{enumerate}
\item[\rm (i)]If $wb=bw$ and $ac=ca$, then $aw\in R^{(b, c)}$ and
$(aw)^{(b, c)}=w^{(b, c)} a^{(b, c)}$.


\item[\rm (ii)] If $ab=ba$ and $ac=ca$, then $aw, wa\in R^{(b,
c)}$, $(aw)^{(b, c)}=w^{(b, c)} a^{(b, c)}$ and $(wa)^{(b,
c)}=a^{(b, c)} w^{(b, c)}$.

\end{enumerate}
\end{cor}
\proof     
(i)   Since  $wb=bw$ and $ac=ca$,  we obtain
             $$w^{(b, c)} a^{(b, c)} awb=w^{(b, c)} a^{(b, c)} a(bw)=w^{(b, c)} (a^{(b, c)} ab)w=w^{(b, c)} bw=w^{(b, c)} wb=b,$$
             $$caww^{(b, c)} a^{(b, c)}=(ac)ww^{(b, c)} a^{(b, c)}=a(cww^{(b, c)}) a^{(b, c)}=aca^{(b, c)}=caa^{(b, c)}=c.$$
             Therefore, by Theorem \ref{ke-reverse(b, c)-th1} (for $s=c$ and $t=b$), $aw\in R^{(b, c)}$ and $(aw)^{(b, c)}=w^{(b, c)} a^{(b, c)}$.

(ii) First, we will prove that $aw\in R^{(b, c)}$ and $(aw)^{(b,
c)}=w^{(b, c)} a^{(b, c)}$. By the proof of (i), the condition
$ac=ca$ gives $c=caww^{(b, c)}a^{(b, c)}$.  Note that from Lemma
\ref{ke-(b,c)-4-rem1}, if $ab=ba$ and $ac=ca$, then $a^{(b, c)}
a=aa^{(b, c)}$. Next we will prove that $w^{(b, c)} a^{(b,
c)}awb=b$. Indeed,
$$\begin{array}{ccl}
w^{(b, c)} a^{(b, c)}awb&=&(w^{(b, c)} c^-c)a^{(b, c)} awb=w^{(b, c)} c^- c (aa^{(b, c)})wb\\
&=&w^{(b, c)}c^- (caa^{(b, c)})wb=w^{(b, c)} c^-cwb\\
&=&w^{(b, c)} wb=b.\end{array}$$

Similarly, we can prove that $wa\in R^{(b, c)}$ and $(wa)^{(b, c)}=a^{(b, c)} w^{(b, c)}$.
\qed

If $b=c=d$, we have the following reverse order law for the Mary inverse.

\begin{cor}{\rm \cite[Theorem 2.14]{Zhu-Chen-Patricio-arxiv}}
Let  $a, w, d\in R$ be such that $a^{\|
d}$ and $w^{\| d}$ exist. If $ad=da$, then $aw, wa\in R^{\| d}$,
$(aw)^{\| d}=w^{\| d} a^{\| d}$ and $(wa)^{\| d}=a^{\| d} w^{\|
d}$.
\end{cor}

\medskip
Next, we consider the following reverse order law
     $$(aw)^{(b, c)}=(a^{(t, c)}aw)^{(b, s)}a^{(t, c)}.$$

\begin{thm}\label{ke-reverse(b, c)-th2}
Let $a, w, b, c, s, t\in R$ be such that $a^{(t, c)}$ and $(a^{(t,
c)}aw)^{(b, s)}$ exist. Then the following statements are
equivalent:
\begin{enumerate}
\item[\rm (i)] $aw\in R^{(b, c)}$ and $(aw)^{(b, c)}=(a^{(t,
c)}aw)^{(b, s)}a^{(t, c)}$;

\item[\rm (ii)] $c=caw(a^{(t, c)}aw)^{(b, s)}a^{(t, c)}$.
\end{enumerate}
\end{thm}
\proof Using Lemma \ref{ke-reverse(b, c)-lem2}, $aw\in R^{(b, c)}$ and $(aw)^{(b, c)}=(a^{(t, c)}aw)^{(b, s)}a^{(t, c)}$ hold if and only if
$$b, c\in R^-, \quad (a^{(t, c)}aw)^{(b, s)}a^{(t, c)}=bb^-(a^{(t, c)}aw)^{(b, s)}a^{(t, c)}=(a^{(t, c)}aw)^{(b, s)}a^{(t, c)}c^-c,$$
and $$b=(a^{(t, c)}aw)^{(b, s)}a^{(t, c)}awb, \quad c=caw(a^{(t, c)}aw)^{(b, s)}a^{(t, c)}.$$
   Since $a^{(t, c)}$ and $(a^{(t, c)}aw)^{(b, s)}$ exist, again using Lemma \ref{ke-reverse(b, c)-lem2}, we have  $b, c\in R^-$,
   $$(a^{(t, c)}aw)^{(b, s)}=bb^-(a^{(t, c)}aw)^{(b, s)}, \quad a^{(t, c)}=a^{(t, c)}c^-c,\quad b=(a^{(t, c)}aw)^{(b, s)}a^{(t, c)}awb.$$
   Therefore, (i) is equivalent to (ii).
\qed

\begin{cor}
Let $a, w, b, c, s, t\in R$ be such that $a^{(t, c)}$ and $(a^{(t, c)}aw)^{(b, s)}$ exist.
\begin{enumerate}
\item[\rm (i)]If $Ra\subseteq Rs$, then $aw\in R^{(b, c)}$ and
$(aw)^{(b, c)}=(a^{(t, c)}aw)^{(b, s)}a^{(t, c)}$.

\item[\rm (ii)]If $aw \in R^{(b, c)}$  and $(a^{(t, c)}aw)^{(b,
s)}=(aw)^{(b, c)}a$, then $(aw)^{(b, c)}=(a^{(t, c)}aw)^{(b,
s)}a^{(t, c)}$.
\end{enumerate}
\end{cor}
\proof
   (i) Similar discuss as Corollary \ref{ke-reverse(b, c)-cor1}, we get $a=as^-s$. As $a^{(t, c)}$ and $(a^{(t, c)}aw)^{(b, s)}$ exist, we know that $caa^{(t, c)}=c$ and $sa^{(t, c)}aw(a^{(t, c)}aw)^{(b, s)}=s$. Therefore,
    $$\begin{array}{ccl}
    caw(a^{(t, c)}aw)^{(b, s)}a^{(t, c)}&=&(caa^{(t, c)})aw(a^{(t, c)}aw)^{(b, s)}a^{(t, c)}\\
    &=&c(as^-s)a^{(t, c)}aw(a^{(t, c)}aw)^{(b, s)}a^{(t, c)}\\
    &=&cas^-(sa^{(t, c)}aw(a^{(t, c)}aw)^{(b, s)})a^{(t, c)}\\
    &=&cas^-sa^{(t, c)}=caa^{(t, c)}=c.
    \end{array}$$
    Using Theorem \ref{ke-reverse(b, c)-th2}, $aw\in R^{(b, c)}$ and $(aw)^{(b, c)}=(a^{(t, c)}aw)^{(b, s)}a^{(t, c)}$.

  (ii) Since $a^{(t, c)}$ and $(aw)^{(b, c)}$ exist, we get $caa^{(t, c)}=c$ and $caw(aw)^{(b, c)}=c$. So the hypothesis  $(a^{(t, c)}aw)^{(b, s)}=(aw)^{(b, c)}a$ gives
       $$caw(a^{(t, c)}aw)^{(b, s)}a^{(t, c)}=caw((aw)^{(b, c)}a)a^{(t, c)}=(caw(aw)^{(b, c)})aa^{(t, c)}=caa^{(t, c)}=c.$$
Hence, by Theorem \ref{ke-reverse(b, c)-th2}, $(aw)^{(b, c)}=(a^{(t, c)}aw)^{(b, s)}a^{(t, c)}$.
\qed

\medskip
In a similar way we can prove the related results for $(aw)^{(b, c)}=w^{(b, s)}(aww^{(b, s)})^{(t, c)}$.

\begin{thm}\label{ke-reverse(b, c)-th3}
Let $a, w, b, c, s, t\in R$ be such that $(aww^{(b, s)})^{(t, c)}$ and $w^{(b, s)}$ exist. Then the following are equivalent:
\begin{enumerate}
\item[\rm (i)] $aw\in R^{(b, c)}$ and $(aw)^{(b, c)}=w^{(b,
s)}(aww^{(b, s)})^{(t, c)}$;

\item[\rm (ii)] $b=w^{(b, s)}(aww^{(b, s)})^{(t, c)}awb$.
\end{enumerate}
\end{thm}

\begin{cor}
Let $a, w, b, c, s, t\in R$ be such that $(aww^{(b, s)})^{(t, c)}$ and $w^{(b, s)}$ exist.
\begin{enumerate}
\item[\rm (i)] If $wR\subseteq tR$, then $aw\in R^{(b, c)}$ and
$(aw)^{(b, c)}=w^{(b, s)}(aww^{(b, s)})^{(t, c)}$.

\item[\rm (ii)] If $aw\in R^{(b, c)}$ and $(aww^{(b, s)})^{(t,
c)}=w(aw)^{(b, c)}$, then $(aw)^{(b, c)}=w^{(b, s)}(aww^{(b,
s)})^{(t, c)}$.
\end{enumerate}
\end{cor}

Also, if $s=c$ and $t=b$, we get the following result.
%
%
%

\begin{cor}
Let $a, w, b, c\in R$.
\begin{enumerate}
\item[\rm (i)] Suppose that $a^{(b, c)}$ and $(a^{(b, c)}aw)^{(b,
c)}$ exist. If  $ac=ca$, then $aw\in R^{(b, c)}$ and $(aw)^{(b,
c)}=(a^{(b, c)}aw)^{(b, c)}a^{(b, c)}$.

\item[\rm (ii)] Suppose that $(aww^{(b, c)})^{(b, c)}$ and $w^{(b,
c)}$ exist. If $wb=bw$,
     then  $aw\in R^{(b, c)}$ and $(aw)^{(b, c)}=w^{(b, c)}(aww^{(b, c)})^{(b, c)}$.
\end{enumerate}
\end{cor}
\proof (i) Since $a^{(b, c)}$ and $(a^{(b, c)}aw)^{(b, c)}$ exist,
we get $caa^{(b, c)}=c=c(a^{(b, c)}aw)(a^{(b, c)}aw)^{(b, c)}.$ If
$ac=ca$, then
 $$\begin{array}{ccl}
 caw(a^{(b, c)}aw)^{(b, c)}a^{(b, c)}&=&(caa^{(b, c)})aw(a^{(b, c)}aw)^{(b, c)}a^{(b, c)}\\
 &=&aca^{(b, c)}aw(a^{(b, c)}aw)^{(b, c)}a^{(b, c)}\\
 &=&aca^{(b, c)}=caa^{(b, c)}=c.
 \end{array}$$
By Theorem \ref{ke-reverse(b, c)-th2} (for $s=c$ and $t=b$),
$aw\in R^{(b, c)}$ and $(aw)^{(b, c)}=(a^{(b, c)}aw)^{(b,
c)}a^{(b, c)}$.

(ii) Similarly as (i). \qed

\begin{cor}
Let $a, w, b, c\in R$.
\begin{enumerate}
\item[\rm (i)] Suppose that $a^{\| d}$ and $(a^{\| d}aw)^{\| d}$
exist. If  $ad=da$, then $aw\in R^{\| d}$ and $(aw)^{\| d}=(a^{\|
d}aw)^{\| d}a^{\| d}$.

\item[\rm (ii)] Suppose that $(aww^{\| d})^{\| d}$ and $w^{\| d}$
exist. If $wd=dw$,
     then  $aw\in R^{\| d}$ and $(aw)^{\| d}=w^{\| d}(aww^{\| d})^{\| d}$.
\end{enumerate}
\end{cor}

\medskip
Now the following reverse order law  is studied
$$(a^{(t, c)}aw)^{(b, u)}=w^{(b, s)}a^{(t, c)}a.$$

\begin{thm}\label{ke-reverse(b, c)-th4}
Let $a, w, b, c, s, t, u\in R$ be such that $a^{(t, c)}$ and
$w^{(b, s)}$ exist. Then the following statements are equivalent:
\begin{enumerate}
\item[\rm (i)] $a^{(t, c)}aw\in R^{(b, u)}$ and $(a^{(t,
c)}aw)^{(b, u)}=w^{(b, s)}a^{(t, c)}a$;

\item[\rm (ii)] $u\in R^-$, $w^{(b, s)}a^{(t, c)}a(1-uu^-)=0$,
$b=w^{(b, s)}a^{(t, c)}awb$ and $c=ca^{(t, c)}aww^{(b, s)}a^{(t,
c)}a.$
\end{enumerate}
\end{thm}
\proof
     Using Lemma \ref{ke-reverse(b, c)-lem2}, (i) holds if and only if $b, u\in R^-$,
     $$w^{(b, s)}a^{(t, c)}a=bb^-w^{(b, s)}a^{(t, c)}a=w^{(b, s)}a^{(t, c)}auu^-,$$
     $$b=w^{(b, s)}a^{(t, c)}aa^{(t, c)}awb, \quad \mbox{and} \quad c=ca^{(t, c)}aww^{(b, s)}a^{(t, c)}a.$$
     Since $a^{(t, c)}$ and $w^{(b, s)}$ exist, it follows that $b\in R^-$, $bb^-w^{(b, s)}=w^{(b, s)}$ and $a^{(t, c)}aa^{(t, c)}=a^{(t, c)}$. Therefore, (i) is equivalent to (ii).
\qed

\medskip
Using Theorem \ref{ke-reverse(b, c)-th4}, we get some sufficient
conditions for $(aw)^{(b, c)}=w^{(b, s)}a^{(t, c)}.$

\begin{cor}\label{ke-reverse(b, c)-cor4}
Let $a, w, b, c, s, t, u\in R$ be such that $a^{(t, c)}$, $w^{(b, s)}$, $(aw)^{(b, c)}$ and $(a^{(t, c)}aw)^{(b, u)}$ exist. If $(aw)^{(b, c)}=(a^{(t, c)}aw)^{(b, u)}a^{(t, c)}$ and one of the following equivalent statements holds:
\begin{enumerate}
\item[\rm (i)]  $(a^{(t, c)}aw)^{(b, u)}=w^{(b, s)}a^{(t, c)}a;$
\item[\rm (ii)] $w^{(b, s)}a^{(t, c)}a(1-uu^-)=0$, $b=w^{(b,
s)}a^{(t, c)}awb$ and $c=ca^{(t, c)}aww^{(b, s)}a^{(t, c)}a,$
\end{enumerate}
then $(aw)^{(b, c)}=w^{(b, s)}a^{(t, c)}.$
\end{cor}
\proof Since $(a^{(t, c)}aw)^{(b, u)}$ exists, we have $u\in R^-$, so (i) is equivalent to (ii) by Theorem \ref{ke-reverse(b, c)-th4}. Therefore, $$(aw)^{(b, c)}=(a^{(t, c)}aw)^{(b, u)}a^{(t, c)}=w^{(b, s)}a^{(t, c)}aa^{(t, c)}=w^{(b, s)}a^{(t, c)}.$$
\qed

Note that if $u=s$ in Corollary \ref{ke-reverse(b, c)-cor4}, using Theorem \ref{ke-reverse(b, c)-th2}, then we can replace the condition $(aw)^{(b, c)}=(a^{(t, c)}aw)^{(b, u)}a^{(t, c)}$ by $c=caw(a^{(t, c)}aw)^{(b, s)}a^{(t, c)}$ to obtain another sufficient condition for $(aw)^{(b, c)}=w^{(b, s)}a^{(t, c)}.$ Here we left it to reader.

\medskip
Analogously, we have the following results.
\begin{thm}\label{ke-reverse(b, c)-th5}
Let $a, w, b, c, s, t, v\in R$ be such that $a^{(t, c)}$ and
$w^{(b, s)}$ exist. Then the following statements are equivalent:
\begin{enumerate}
\item[\rm (i)]$aww^{(b, s)}\in R^{(v, c)}$ and $(aww^{(b,
s)})^{(v, c)}=ww^{(b, s)}a^{(t, c)}$;
\item[\rm (ii)]$v\in R^-$,
$(1-vv^-)ww^{(b, s)}a^{(t, c)}=0$, $b=ww^{(b, s)}a^{(t,
c)}aww^{(b, s)}b$ and $c=caww^{(b, s)}a^{(t, c)}$.
\end{enumerate}
\end{thm}

\begin{cor}
Let $a, w, b, c, s, t, v\in R$ be such that $a^{(t, c)}$, $w^{(b, s)}$, $(aw)^{(b, c)}$ and $(aww^{(b, s)})^{(v, c)}$ exist. If $(aw)^{(b, c)}=w^{(b, s)}(aww^{(b, s)})^{(t, c)}$ and one of the following equivalent statements holds:
\begin{enumerate}
\item[\rm (i)]$aww^{(b, s)}\in R^{(v, c)}$ and $(aww^{(b,
s)})^{(v, c)}=ww^{(b, s)}a^{(t, c)}$;

\item[\rm (ii)]$(1-vv^-)ww^{(b, s)}a^{(t, c)}=0$, $b=ww^{(b,
s)}a^{(t, c)}aww^{(b, s)}b$ and $c=caww^{(b, s)}a^{(t, c)}$,
\end{enumerate}
then $(aw)^{(b, c)}=w^{(b, s)}a^{(t, c)}.$
\end{cor}


\begin{cor}
Let $a, w, b, c, u, v\in R$ be such that $a^{(b, c)}$ and $w^{(b, c)}$ exist. Then
\begin{enumerate}
\item[\rm (i)] If $wb=bw$, then  $a^{(b, c)}aw\in R^{(b, u)}$ and
$(a^{(b, c)}aw)^{(b, u)}=w^{(b, c)}a^{(b, c)}a$ if and only if
$$u\in R^-, \quad w^{(b, c)}a^{(b, c)}a(1-uu^-)=0, \quad c=ca^{(b, c)}a.$$

\item[\rm (ii)] If $ca=ac$, then $aww^{(b, c)}\in R^{(v, c)}$ and
$(aww^{(b, c)})^{(v, c)}=ww^{(b, c)}a^{(b, c)}$ if and only if
$$v\in R^-, \quad (1-vv^-)ww^{(b, c)}a^{(b, c)}=0, \quad b=ww^{(b, c)}b.$$
\end{enumerate}
\end{cor}

\begin{proof} (i) It follows by Theorem \ref{ke-reverse(b, c)-th4}
and
\begin{eqnarray*}
ca^{(b, c)}aww^{(b, c)}a^{(b, c)}a&=&ca^{(b, c)}a(wb)b^-w^{(b,
c)}a^{(b, c)}a=c(a^{(b, c)}ab)wb^-w^{(b, c)}a^{(b, c)}a\\
&=&c(bw)b^-w^{(b, c)}a^{(b, c)}a=cw(bb^-w^{(b, c)})a^{(b, c)}a\\
&=&cww^{(b, c)}a^{(b, c)}a=ca^{(b, c)}a.
\end{eqnarray*}

(ii) In the similar way as (i).
\end{proof}

\medskip
Next, we investigate the following reverse order law
$$(aw)^{(b, c)}=w^{(b, c)}(a^{(t, c)}aww^{(b, c)})^{(b,c)}a^{(t, c)}.$$

\begin{thm}
Let $a, w, b, c, t\in R$ be such that $a^{(t, c)}$, $w^{(b, c)}$
and $(a^{(t, c)}aww^{(b, c)})^{(b,c)}$ exist. Then the following
statements are equivalent:
\begin{enumerate}
\item[\rm (i)] $aw\in R^{(b, c)}$ and $(aw)^{(b, c)}=w^{(b,
c)}(a^{(t, c)}aww^{(b, c)})^{(b,c)}a^{(t, c)};$

\item[\rm (ii)] $a^{(t, c)}aw\in R^{(b, c)}$, $aww^{(b, c)}\in
R^{(t, c)}$, $(a^{(t, c)}aw)^{(b,c)}=w^{(b,c)}(a^{(t,
c)}aww^{(b,c)})^{(b,c)}$ and \linebreak
$(aww^{(b,c)})^{(t,c)}=(a^{(t, c)}aww^{(b,c)})^{(b,c)}a^{(t,c)}$.
\end{enumerate}
\end{thm}
\proof  (i)$\Rightarrow$ (ii): If (i) holds, first we will observe
that \begin{equation}\label{ke-reverse(b, c)-eq3} a^{(t, c)}aw\in
R^{(b,c)}\quad \mbox{and}\quad (a^{(t, c)}aw)^{(b,c)}=w^{(b,
c)}(a^{(t, c)}aww^{(b,c)})^{(b,c)}.
\end{equation}
In fact,  since $a^{(t, c)}$, $w^{(b, c)}$ and $(a^{(t,
c)}aww^{(b, c)})^{(b,c)}$ exist, by Lemma \ref{ke-reverse(b,
c)-lem2}, we know $b, c, t\in R^-$, $$bb^-w^{(b,c)}=w^{(b,c)},
\quad (a^{(t, c)}aww^{(b,c)})^{(b,c)}=(a^{(t, c)}aww^{(b,
c)})^{(b,c)}c^-c,$$
$$ca^{(t, c)}aww^{(b,c)}(a^{(t, c)}aww^{(b,c)})^{(b,c)}=c.$$
Then $w^{(b,c)}(a^{(t, c)}aww^{(b,c)})^{(b,c)}=bb^-w^{(b,
c)}(a^{(t, c)}aww^{(b,c)})^{(b,c)}=w^{(b,c)}(a^{(t,c)}aww^{(b,
c)})^{(b,c)}c^-c$.
As (i) holds, we have
$$w^{(b,c)}(a^{(t, c)}aww^{(b,c)})^{(b,c)}a^{(t, c)}awb=(aw)^{(b,c)}(aw)b=b.$$
 Again using Lemma \ref{ke-reverse(b, c)-lem2},
we can conclude that (\ref{ke-reverse(b, c)-eq3}) holds.

 Similarly, we can obtain that $aww^{(b,c)}\in R^{(t, c)}$ and
 $(aww^{(b,c)})^{(t, c)}=(a^{(t, c)}aww^{(b,c)})^{(b,c)}a^{(t, c)}$.

        (ii)$\Rightarrow$ (i): If (ii) holds, by Lemma \ref{ke-reverse(b, c)-lem2},  we see $b, c\in R^-$,
        $$w^{(b,c)}(a^{(t, c)}aww^{(b,c)})^{(b,c)}=bb^-w^{(b,c)}(a^{(t, c)}aww^{(b,c)})^{(b,c)},$$
        $$(a^{(t, c)}aww^{(b,c)})^{(b,c)}a^{(t, c)}=(a^{(t, c)}aww^{(b,c)})^{(b,c)}a^{(t, c)}c^-c,$$
        $$w^{(b,c)}(a^{(t, c)}aww^{(b,c)})^{(b,c)}a^{(t, c)}awb=(a^{(t, c)}aw)^{(b, c)}a^{(t, c)}awb=b,$$
        $$caww^{(b,c)}(a^{(t, c)}aww^{(b,c)})^{(b,c)}a^{(t, c)}=caww^{(b,c)}(aww^{(b,c)})^{(t,c)}=c.$$
        Therefore, we can conclude that (i) holds.
\qed


\begin{rem}\rm{
Since the Mary inverse, Bott-Duffin $(e, f)$-inverse, image-kernel
$(p, q)$-inverse are all the special cases of $(b, c)$-inverse if
we choose $b$ and $c$ appropriately, related results for these
inverses are obtained. }
\end{rem}

The reverse order law for the $(p, q)$-outer generalized inverse are considered in the following sequel.

\begin{prop} Let $p, q, r, s\in R^\bullet$ and let $a, b\in R$ be such that $a^{(2)}_{p, q}$ and $b^{(2)}_{s, 1-p}$ exist. Then:
\begin{enumerate}
\item[\rm (i)] $ab\in R^{(2)}_{s, q}$ and $(ab)^{(2)}_{s,
q}=b^{(2)}_{s, 1-p}a^{(2)}_{p, q}$;

\item[\rm (ii)] $a^{(2)}_{p, q}ab\in R^{(2)}_{s, 1-p}$ and
$(a^{(2)}_{p, q}ab)^{(2)}_{s, 1-p}=b^{(2)}_{s, 1-p}a^{(2)}_{p,
q}a$;

\item[\rm (iii)] $abb^{(2)}_{s, 1-p}\in R^{(2)}_{p, q}$ and
$(abb^{(2)}_{s, 1-p})^{(2)}_{p, q}=bb^{(2)}_{s, 1-p}a^{(2)}_{p,
q}$.
\end{enumerate}
\end{prop}

\begin{prop} Let $p, q, r, s, t\in R^\bullet$ and $a, b\in R$.
\begin{enumerate}
\item[\rm (i)] If $a\in R^{(2)}_{p, q}$ and $ab\in R^{(2)}_{t,
q}$, then $a^{(2)}_{p, q}ab\in R^{(2)}_{t, 1-p}$ and $(a^{(2)}_{p,
q}ab)^{(2)}_{t, 1-p}=(ab)^{(2)}_{t, q}a$.

\item[\rm (ii)] If $b\in R^{(2)}_{s, t}$ and $ab\in R^{(2)}_{s,
r}$, then $abb^{(2)}_{s, t}\in R^{(2)}_{1-t, r}$ and
$(abb^{(2)}_{s, t})^{(2)}_{1-t, r}=b(ab)^{(2)}_{s, r}$.

\item[\rm (iii)] If $a\in R^{(2)}_{p, q}$, $b\in R^{(2)}_{s, t}$
and $ab\in R^{(2)}_{s, q}$, then $a^{(2)}_{p, q}abb^{(2)}_{s,
t}\in R^{(2)}_{1-t, 1-p}$ and $(a^{(2)}_{p, q}abb^{(2)}_{s,
t})^{(2)}_{1-t, 1-p}=b(ab)^{(2)}_{s, q}a$.
\end{enumerate}
\end{prop}

\begin{prop} Let $p, q, r, s, t\in R^\bullet$ and $a, b\in R$.
\begin{enumerate}
\item[\rm (i)] If $a\in R^{(1, 2)}_{p, q}$ and $a^{(1, 2)}_{p,
q}ab\in R^{(2)}_{t, 1-p}$, then $ab\in R^{(2)}_{t, q}$ and
$(ab)^{(2)}_{t, q}=(a^{(1, 2)}_{p, q}ab)^{(2)}_{t, 1-p}a^{(1,
2)}_{p, q}$.

\item[\rm (ii)] If $b\in R^{(1, 2)}_{s, t}$ and $abb^{(1, 2)}_{s,
t}\in R^{(2)}_{1-t, r}$, then $ab\in R^{(2)}_{s, r}$ and
$(ab)^{(2)}_{s, r}=b^{(1, 2)}_{s, t}(abb^{(1, 2)}_{s,
t})^{(2)}_{1-t, r}$.

\item[\rm (iii)] If $a\in R^{(1, 2)}_{p, q}$, $b\in R^{(1, 2)}_{s,
t}$ and $a^{(1, 2)}_{p, q}abb^{(1, 2)}_{s, t}\in R^{(2)}_{1-t,
1-p}$, then $ab\in R^{(2)}_{s, q}$ and $(ab)^{(2)}_{s, q}=b^{(1,
2)}_{s, t}(a^{(1, 2)}_{p, q}abb^{(1, 2)}_{s, t})^{(2)}_{1-t,
1-p}a^{(1, 2)}_{p, q}$.
\end{enumerate}
\end{prop}

Finally, we consider the more general case of reverse order law
$(a_1a_2)^{(b_3, c_3)}=a_2^{(b_2, c_2)}a_1^{(b_1, c_1)}$ when
$a_i$ is $(b_i, c_i)$-invertible with $(b_i, c_i)$-inverse
$a_i^{(b_i, c_i)}~(i=1, 2)$. Before to discuss it, we need the
following lemma which can be seen in \cite{Ke-dragana}. We will
give the proof of the following result for the sake of
completeness.
\begin{lem}\emph{\cite[Theorem 2.11]{Ke-dragana}}\label{ke-(b,c)-4-lem0-ke3th2.11}
Let $a, b, c\in R$ be such that $a^{(b, c)}$ exists. Let $w\in R$ be such that $wR=bR$  and $w^\circ=c^\circ$ (or $Rw=Rc$). Then $aw, wa \in R^\#$, $a^{(b, c)}=w(aw)^\#=(wa)^\#w$, and $w=a^{(b, c)} aw=waa^{(b, c)}$.
\end{lem}
\proof   If $a^{(b, c)}$ exists, by Lemma \ref{ke-reverse(b, c)-lem2}, $b, c\in R^-$. The condition $b\in R^-$ and  $wR=bR$ imply that $w\in R^-$. So $w^\circ=c^\circ$ is equivalent to $Rw=Rc$.

     By Lemma \ref{drazinLAA12-thm2.2}, we have $a^{(b, c)}aa^{(b, c)}=a^{(b, c)}$, $a^{(b, c)}R=bR=wR$ and $Ra^{(b, c)}=Rc=Rw$. Then there exist $x, y, z\in R$ such that $w=a^{(b, c)}x=ya^{(b, c)}$ and $a^{(b, c)}=wz$. So $w=a^{(b, c)}aa^{(b, c)}x=a^{(b, c)}aw$. Similarly, $w=waa^{(b, c)}$.

     Let $u\in (aw)^\circ$. We have $awu=0$, which gives $wu=a^{(b, c)}awu=0$, i.e., $u\in w^\circ$. Hence $(aw)^\circ=w^\circ$.
     Since $a^{(b, c)}$ exists, using \cite[Proposition 2.7]{D}, $R=abR\oplus c^\circ$. Thus $R=awR\oplus (aw)^\circ$.
    From the proof of \cite[Proposition 7, p.205]{Hartwig-1976block}, $aw\in R^\#$.

    Let $v=w(aw)^\#$. We can show that $v$ is the $(b, c)$-inverse of $a$. Indeed, as $bR=wR$ and $Rw=Rc$, there are $x_1, x_2, y_1, y_2\in R$ such that $b=wx_1$, $w=by_1=x_2c$ and $c=y_2w$. Thus,
    $$bb^-v=bb^-w(aw)^\#=bb^-by_1(aw)^\#=by_1(aw)^\#=w(aw)^\#=v.$$
    Similarly, we get $vc^-c=v$. Since $aw((aw)^\#aw-1)=0$ and $(aw)^\circ=w^\circ$, we have $w=w(aw)^\#aw$. Hence, $$vab=w(aw)^\#ab=w(aw)^\#awx_1=wx_1=b,$$ $$cav=caw(aw)^\#=y_2waw(aw)^\#=y_2w=c.$$
    Therefore, by Lemma \ref{ke-reverse(b, c)-lem2}, $v=a^{(b, c)}$.

    Similarly, we can prove that $wa\in R^\#$ and $a^{(b, c)}=(wa)^\#w$.
\qed

\begin{thm}\label{ke-reverse(b, c)-th-main}
Let $a_i, b_i, c_i, b_3, c_3 \in R~(i=1, 2)$ be such that $a_i$ is
$(b_i, c_i)$-invertible with $(b_i, c_i)$-inverse $a_i^{(b_i,
c_i)}~(i=1, 2)$. Let $a'_1, a_2'\in R$ satisfy
$$a_i'R=b_iR, ~ Ra_i'=Rc_i~(i=1, 2),$$
$$a_2'a_1'R=b_3R, ~ Ra_2'a_1'=Rc_3.$$
If $a_1^{(b_1, c_1)}a_1$ commutes with $a_2a_2'$ and
$a_2a_2^{(b_2, c_2)}$ commutes with $a_1'a_1$, then $a_1a_2$ is
$(b_3, c_3)$-invertible and
$$(a_1a_2)^{(b_3, c_3)}=a_2^{(b_2, c_2)}a_1^{(b_1, c_1)}.$$
\end{thm}
\proof  First, if $a_i$ is $(b_i, c_i)$-invertible with $(b_i,
c_i)$-inverse $a_i^{(b_i, c_i)}~(i=1, 2)$, by Lemma
\ref{drazinLAA12-thm2.2}, it follows that
 \begin{equation}\label{ke-(b,c)-4-eq1}
 a_i^{(b_i, c_i)}a_ia_i^{(b_i, c_i)}=a_i^{(b_i, c_i)},~ a_i^{(b_i, c_i)}R=b_iR,~ Ra_i^{(b_i, c_i)}=Rc_i ~(i=1, 2).
 \end{equation}
Since $a_i^{(b_i, c_i)}~(i=1, 2)$ exist and $a_i'R=b_iR, ~ Ra_i'=Rc_i~(i=1, 2)$, as an application of Lemma \ref{ke-(b,c)-4-lem0-ke3th2.11},
 we have $a_ia_i', a_i'a_i~(i=1, 2)\in R^\#$, and
\begin{equation}\label{ke-(b,c)-4-eq1-group}
a_i^{(b_i, c_i)}=a_i'(a_ia_i')^\#=(a_i'a_i)^\#a_i',~~\mbox{and}~~ a_i'=a_i^{(b_i, c_i)}a_i a_i'=a_i'a_ia_i^{(b_i, c_i)} ~(i=1, 2).
\end{equation}
As $a_1'a_1,~ a_2a_2'\in R^\#$,  $a_1^{(b_1, c_1)}a_1$
commutes with $a_2a_2'$ and $a_2a_2^{(b_2, c_2)}$ commutes with
$a_1'a_1$, by Lemma \ref{ke-(b,c)-4-lem6-commut}, we obtain the
following equations
\begin{equation}\label{ke-(b,c)-4-eq2-group-commu}
a_1^{(b_1, c_1)}a_1(a_2a_2')^\#=(a_2a_2')^\#a_1^{(b_1, c_1)}a_1,
\end{equation}
\begin{equation}\label{ke-(b,c)-4-eq2'-group-commu}
a_2a_2^{(b_2, c_2)}(a_1'a_1)^\#=(a_1'a_1)^\#a_2 a_2^{(b_2, c_2)}.
\end{equation}
Therefore, using equations (\ref{ke-(b,c)-4-eq1-group}) and (\ref{ke-(b,c)-4-eq2-group-commu}) we have
\begin{equation}\label{ke-(b,c)-4-eq1-(b,c)-commu}
\begin{split}\begin{array}{ccl}
a_1^{(b_1, c_1)}a_1a_2a_2^{(b_2, c_2)}
&\overset{\tiny{(\ref{ke-(b,c)-4-eq1-group})}}=&
a_1^{(b_1, c_1)}a_1a_2(a_2'(a_2a_2')^\#)
=(a_1^{(b_1, c_1)}a_1a_2a_2')(a_2a_2')^\#\\
&=&(a_2a_2'a_1^{(b_1, c_1)}a_1)(a_2a_2')^\#
=a_2a_2'(a_1^{(b_1, c_1)}a_1(a_2a_2')^\#)\\
&\overset{\tiny{(\ref{ke-(b,c)-4-eq2-group-commu})}}=&
a_2a_2'((a_2a_2')^\#a_1^{(b_1, c_1)}a_1)
\overset{\tiny{(\ref{ke-(b,c)-4-eq1-group})}}
=a_2a_2^{(b_2, c_2)}a_1^{(b_1, c_1)}a_1.
\end{array}\end{split}\end{equation}

Now we will  prove that $a_1a_2\in R^{(b_3, c_3)}$ and $(a_1a_2)^{(b_3, c_3)}=a_2^{(b_2, c_2)}a_1^{(b_1, c_1)}.$ If we  prove that
\begin{equation}\label{ke-(b,c)-4-eq3}
(a_2^{(b_2, c_2)}a_1^{(b_1, c_1)})(a_1a_2)(a_2^{(b_2, c_2)}a_1^{(b_1, c_1)})=a_2^{(b_2, c_2)}a_1^{(b_1, c_1)},
\end{equation}
\begin{equation}\label{ke-(b,c)-4-eq4}
a_2^{(b_2, c_2)}a_1^{(b_1, c_1)}R=b_3R,\quad Ra_2^{(b_2, c_2)}a_1^{(b_1, c_1)}=Rc_3.
\end{equation}
Then by Lemma \ref{drazinLAA12-thm2.2}, the equation
$(a_1a_2)^{(b_3, c_3)}=a_2^{(b_2, c_2)}a_1^{(b_1, c_1)}$ hold.
Indeed,
$$\begin{array}{ccl}
(a_2^{(b_2, c_2)}a_1^{(b_1, c_1)})(a_1a_2)(a_2^{(b_2, c_2)}a_1^{(b_1, c_1)})&=&
a_2^{(b_2, c_2)}(a_1^{(b_1, c_1)}a_1a_2a_2^{(b_2, c_2)})a_1^{(b_1, c_1)}\\
&\overset{\tiny{(\ref{ke-(b,c)-4-eq1-(b,c)-commu})}}=&a_2^{(b_2, c_2)}(a_2a_2^{(b_2, c_2)}a_1^{(b_1, c_1)}a_1)a_1^{(b_1, c_1)}\\
&=&(a_2^{(b_2, c_2)}a_2a_2^{(b_2, c_2)})(a_1^{(b_1, c_1)}a_1a_1^{(b_1, c_1)})\\
&\overset{\tiny{(\ref{ke-(b,c)-4-eq1})}}=&a_2^{(b_2,
c_2)}a_1^{(b_1, c_1)},
\end{array}$$
that is, the equation (\ref{ke-(b,c)-4-eq3}) holds.

Since $a_i^{(b_i, c_i)}a_ia_i^{(b_i, c_i)}=a_i^{(b_i, c_i)}$, we have $a_i^{(b_i, c_i)}R=a_i^{(b_i, c_i)}a_iR$, $R a_i a_i^{(b_i, c_i)}=Ra_i^{(b_i, c_i)}~(i=1,2)$. Thus,
$$\begin{array}{ccl}
a_2^{(b_2, c_2)}a_1^{(b_1, c_1)}R
&=&a_2^{(b_2, c_2)}a_1^{(b_1, c_1)}a_1R
\overset{\tiny{(\ref{ke-(b,c)-4-eq1-group})}}=
(a_2'(a_2a_2')^\#)a_1^{(b_1, c_1)}a_1R
=a_2'((a_2a_2')^\#a_1^{(b_1, c_1)}a_1)R\\
&\overset{\tiny{(\ref{ke-(b,c)-4-eq2-group-commu})}}=&
a_2'(a_1^{(b_1, c_1)}a_1(a_2a_2')^\#)R
\overset{\tiny{(\ref{ke-(b,c)-4-eq1-group})}}=
a_2'(a_1'(a_1a_1')^\#)a_1(a_2a_2')^\#R \\
&\subseteq& a_2'a_1'R=a_2'b_1R\overset{\tiny{(\ref{ke-(b,c)-4-eq1})}}=
a_2'a_1^{(b_1, c_1)}R=a_2'a_1^{(b_1, c_1)}a_1R\\
&\overset{\tiny{(\ref{ke-(b,c)-4-eq1-group})}}=&
(a_2^{(b_2, c_2)}a_2a_2')a_1^{(b_1, c_1)}a_1R=a_2^{(b_2, c_2)}(a_2a_2'a_1^{(b_1, c_1)}a_1)R\\
&=&a_2^{(b_2, c_2)}(a_1^{(b_1, c_1)}a_1 a_2a_2')R\\
&\subseteq&a_2^{(b_2, c_2)}a_1^{(b_1, c_1)}R.
\end{array}$$
Consequently, $a_2^{(b_2, c_2)}a_1^{(b_1, c_1)}R=a_2'a_1'R=b_3R$,
i.e. the left equation of (\ref{ke-(b,c)-4-eq4}) holds.

Similarly, we get
$$\begin{array}{ccl}
Ra_2^{(b_2, c_2)}a_1^{(b_1, c_1)}&=&Ra_2a_2^{(b_2, c_2)}a_1^{(b_1, c_1)}\overset{\tiny{(\ref{ke-(b,c)-4-eq1-group})}}=
Ra_2a_2^{(b_2, c_2)}((a_1'a_1)^\#a_1')=R(a_2a_2^{(b_2, c_2)}(a_1'a_1)^\#)a_1'\\
&\overset{\tiny{(\ref{ke-(b,c)-4-eq2'-group-commu})}}=&
R((a_1'a_1)^\#a_2a_2^{(b_2, c_2)})a_1'
\overset{\tiny{(\ref{ke-(b,c)-4-eq1-group})}}=
R(a_1'a_1)^\#a_2((a_2'a_2)^\#a_2')a_1'\\
&\subseteq& Ra_2'a_1'=Rc_2a_1'\overset{\tiny{(\ref{ke-(b,c)-4-eq1})}}
=Ra_2^{(b_2, c_2)}a_1'=Ra_2a_2^{(b_2, c_2)}a_1'\\
&\overset{\tiny{(\ref{ke-(b,c)-4-eq1-group})}}=&
Ra_2a_2^{(b_2, c_2)}(a_1'a_1a_1^{(b_1, c_1)})
=R(a_2a_2^{(b_2, c_2)}a_1'a_1)a_1^{(b_1, c_1)}\\
&=&R(a_1'a_1 a_2a_2^{(b_2, c_2)})a_1^{(b_1, c_1)}\\
&\subseteq&Ra_2^{(b_2, c_2)}a_1^{(b_1, c_1)}.
\end{array}$$
Thus, $Ra_2^{(b_2, c_2)}a_1^{(b_1, c_1)}=Ra_2'a_1'=Rc_3$, i.e. the
right equation of (\ref{ke-(b,c)-4-eq4}) holds.
\qed

Let $b_i=c_i=d_i~(i=1, 2, 3)$ in Theorem \ref{ke-reverse(b,
c)-th-main}, then we have the following result for Mary inverse.

\begin{cor}
Let $a_i, d_i \in R~(i=1, 2)$ be such that $a_i^{\| d_i}~(i=1,2)$
exists. If there exists $d_3\in R$ such that $d_2d_1R=d_3R, ~
Rd_2d_1=Rd_3,$ and  $a_1^{\| d_1}a_1$ commutes with $a_2d_2$ and
$a_2a_2^{\| d_2}$ commutes with $d_1a_1$, then $a_1a_2\in R^{\|
d_3}$  and
$$(a_1a_2)^{\| d_3}=a_2^{\| d_2}a_1^{\| d_1}.$$
\end{cor}
\proof  By \cite[Theorem 7]{Mary2011}, we know that if $a^{\| d}$ exists, then $ad, da\in R^\#$ and $a^{\| d}=d(ad)^\#=(da)^\#$. Moreover, by the definition of Mary inverse, we obtain that $d=a^{\| d}ad=daa^{\| d}$, $a^{\| d}R=dR$ and $Ra^{\| d}=Rd$.
Take $a_i'=d_i~(i=1, 2)$ and $b_i=c_i=d_i~(i=1, 2, 3)$ in Theorem \ref{ke-reverse(b, c)-th-main}, as required.
\qed

\medskip
As an application of this theorem, we first recall the basic properties of outer generalized inverses with prescribed range and kernel (see \cite{Djordjevic-book08}). Let $X$ and $Y$ be Banach spaces and let $\mathcal{L}(X, Y)$ denote the set of
all bounded operators from $X$ to $Y$. For $A\in \mathcal{L}(X, Y)$ we use $N(A)$ and $R(A)$ to denote the range and the kernel of $A$. We say that $C\in \mathcal{L}(Y, X)$ is an outer generalized inverse of $A$, if $CAC=C$. Let $T$ and $S$ be subspaces of $X$ and $Y$, resp., such that there exists an outer generalized inverse $A^{(2)}_{T, S}\in \mathcal{L}(Y, X)$ of $A$ with range equal to $T$ and kernel equal to $S$, i.e., $A^{(2)}_{T, S}$ satisfies
$$A^{(2)}_{T, S}AA^{(2)}_{T, S}=A^{(2)}_{T, S},\quad R(A^{(2)}_{T, S})=T, \quad N(A^{(2)}_{T, S})=S.$$
If $A, T$ and $S$ given as above, then $A^{(2)}_{T, S}$ exists if and
only if $T$ and $S$, respectively, are closed
and complemented subspaces of $X$ and $Y$, the reduction $A_T:
T\rightarrow A(T)$ is invertible and $A(T)\oplus S = Y$. In this
case $A^{(2)}_{T, S}$ is unique.

The following result can be seen in \cite[Theorem 3.3]{Djordjevic-ActaSM-01}.

\begin{prop}
Let $A \in \mathcal{L}(Y, Z)$, $B \in \mathcal{L}(X, Y)$, and let $A^{(2)}_{M, N}\in \mathcal{L}(Z, Y)$, $B^{(2)}_{T, S}\in \mathcal{L}(Y, X)$ and $(AB)^{(2)}_{K, L}\in \mathcal{L}(Z, X)$ be outer inverses of $A, B$ and $AB$ with subspaces $K, T\subseteq X,$ $M, S\subseteq Y,$ and $N, L\subseteq Z$. Let operators $A'_{M, N}\in \mathcal{L}(Z, Y)$, $B'_{T, S}\in \mathcal{L}(Y, X)$ and $(AB)'_{K, L}\in \mathcal{L}(Z, X)$ satisfy
$$R(A'_{M, N})=M, ~N(A'_{M, N})=N, ~R(B'_{T, S})=T, ~N(B'_{T, S})=S,$$
$$R(B'_{T, S}A'_{M, N})=R((AB)'_{K, L})=K, ~N(B'_{T, S}A'_{M, N})=N((AB)'_{K, L})=L.$$
If $A^{(2)}_{M, N}A$ commutes with $BB'_{T, S}$ and $BB^{(2)}_{T, S}$ commutes with $A'_{M, N}A$, then
$$(AB)^{(2)}_{K, L}=B^{(2)}_{T, S}A^{(2)}_{M, N}.$$
\end{prop}

For the $(p, q)$-outer generalized inverse we have the following result.
\begin{prop}
Let $a, w\in R$ and $e, f, p, q, k, l\in R^\bullet$ be such that
$a_{p, q}^{(2)}$ and $w_{e, f}^{(2)}$ exist. If $a_{p, q}^{(2)}a$ commutes with $ww_{e, f}^{(2)}$,  then $(aw)_{k, l}^{(2)}$ exists and
$(aw)_{k, l}^{(2)}=w_{e, f}^{(2)}a_{p, q}^{(2)}$
if and only if $w_{e, f}^{(2)}pw=k$ and $a(1-f)a_{p, q}^{(2)}=1-l$.
\end{prop}
\proof
   Since $a_{p, q}^{(2)}a$ commutes with $ww_{e, f}^{(2)}$, we have
   $$w_{e, f}^{(2)}a_{p, q}^{(2)}(aw)w_{e, f}^{(2)}a_{p, q}^{(2)}=w_{e, f}^{(2)}(a_{p, q}^{(2)}aww_{e, f}^{(2)})a_{p, q}^{(2)}=w_{e, f}^{(2)}(ww_{e, f}^{(2)}a_{p, q}^{(2)}a)a_{p, q}^{(2)}=w_{e, f}^{(2)}a_{p, q}^{(2)}.$$
   Also,
   $$w_{e, f}^{(2)}a_{p, q}^{(2)}(aw)=w_{e, f}^{(2)} (a_{p, q}^{(2)}a)w=w_{e, f}^{(2)}pw,$$
   $$(aw)w_{e, f}^{(2)}a_{p, q}^{(2)}=a(ww_{e, f}^{(2)})a_{p, q}^{(2)}=a(1-f)a_{p, q}^{(2)}.$$
   Therefore, by the definition of $(p, q)$-outer generalized inverse, we obtain the conclusion.
\qed


\vspace{0.2cm} \noindent {\large\bf Acknowledgments}

           The first author is grateful to China Scholarship Council for supporting her to purse her further study with Professor D. S. Cvetkovi\'{c}-Ili\'{c} in University of Ni\v{s}, Serbia. 
           The research was supported by The National Natural Science Foundation of China (No. 11371089),
           the Specialized Research Fund for the Doctoral Program of Higher Education (No. 20120092110020),
           the Natural Science Foundation of Jiangsu Province (No. BK20141327),
           the Foundation of Graduate Innovation Program of Jiangsu Province (No. KYLX$_{-}$0080).
The second author is supported by the Ministry of Education and
Science, Republic of Serbia, grant no. 174007.

\end{document}